\documentclass[12pt,oneside]{amsart}

\usepackage{geometry}                
\usepackage{graphicx}
\usepackage{amssymb}
\usepackage{epstopdf}
\usepackage{color}
\usepackage{bbm, dsfont,hyperref}

\usepackage{ulem}
\newtheorem{theorem}{Theorem}
\newtheorem{lemma}{Lemma}[section]
\newtheorem{corollary}{Corollary}

\theoremstyle{definition}
\newtheorem{definition}{Definition}
\theoremstyle{remark}
\newtheorem{remark}{Remark}

\theoremstyle{example}

\usepackage{amsmath}
\usepackage{mathrsfs}

\usepackage{verbatim}
\usepackage{amsmath,amssymb,amsthm}             
\usepackage{amssymb}             
\usepackage{amsfonts}            
\usepackage{mathrsfs}            
\usepackage{amsthm}
\usepackage{algorithm}  
\usepackage{algorithmicx}  
\usepackage{algpseudocode}  
\usepackage{mathtools}
\usepackage{commath}
\usepackage{physics}
\usepackage{bm}
\usepackage{graphicx}

\usepackage{float}
\usepackage{listings}
\usepackage{subfigure}
\usepackage{multirow}
\usepackage{diagbox}
\usepackage{color}
\usepackage{bbm}
\usepackage[export]{adjustbox}
\usepackage{enumerate}
\usepackage{bbm}

\usepackage{amsmath}
\usepackage{mathrsfs}

\newcommand{\RNum}[1]{\uppercase\expandafter{\romannumeral #1\relax}}
\begin{document}

\numberwithin{equation}{section}

\title{Some Rigorous Results on the Phase Transition of Finitary Random Interlacement}

\author{Zhenhao Cai}
\address[Zhenhao Cai]{Peking University}
\email{caizhenhao@pku.edu.cn}

\author{Yuan Zhang}
\address[Yuan Zhang]{Peking University}
\email{zhangyuan@math.pku.edu.cn}
\urladdr{https://www.math.pku.edu.cn/teachers/zhangyuan/eindex.html}
\maketitle
	
\begin{abstract}
In this paper, we show several rigorous results on the phase transition of Finitary Random Interlacement (FRI). For the high intensity regime, we show the existence of a  critical fiber length, and give the exact asymptotic of it as intensity goes to infinity. At the same time, our result for the low intensity regime proves the global existence of a non-trivial phase transition with respect to the system intensity.
\end{abstract}
	
\section{Introduction}

The model of finitary random interlacements (FRI) was first introduced by Bowen \cite{bowen2019finitary} to solve a special case
of the Gaboriau-Lyons problem. Intuitively speaking, FRI (denoted by $\mathcal{FI}^{u,T}$) can be seen a ``cloud of twisted yarn" composed of finite `` fibers" on $\mathbb{Z}^d$. The fibers within this system each form a geometrically killed simple random walk, where the starting points are sampled according to a Poisson point process with intensity measure  proportional to system intensity $u$ and inversely proportional to $T+1$. At the same time, the expected length of each fiber is given by $T$ \cite{bowen2019finitary,cai2020chemical,cai2020non,procaccia2019percolation}. See Section \ref{sec2} for precise definitions.

In \cite{bowen2019finitary}, Bowen proved that as $T\to +\infty$ , $\mathcal{FI}^{u,T}$ converges under weak star topology to the celebrated random interlacements (RI) introduced by Sznitman \cite{sznitman2010vacant}, with intensity $u$. FRI can also be seen as a finitary version of RI, in the sense that it can be defined as massive interlacements on a graph equipped with killing measure (see Chapter V of \cite{prevost2020percolation}). 

On the other hand, unlike the classic RI, whose trajectories always form an a.s. connected network \cite{ vcerny2012internal, sznitman2010vacant}, the collection of edges traversed by FRI on  $\mathbb{Z}^d$ has exhibited a non-trivial percolative phase transition \cite{procaccia2019percolation}. And since there are now two parameters to play with, one may characterize such transition from either of the following viewpoints:

(1) One may first fix the intensity $u$ and examine the evolution of FRI with respect to $T$. This is also the setting Bowen used in his first paper \cite{bowen2019finitary}. In addition to the convergence when $T\to\infty$, he also proved that the FRI on non-amenable graphs will a.s. have infinite connected cluster(s) for all sufficiently large $T$. He proposed same question for FRI on $\mathbb{Z}^d$, which was affirmatively answered in \cite{procaccia2019percolation} by proving the existence of the following phase transition with respect to $T$: For any $d\ge 3$ and $u>0$, there is $T_0,T_1\in(0,\infty)$, such that $\mathcal{FI}^{u,T}$ a.s. has no infinite cluster for all $T<T_0$, and a.s. has a unique infinite cluster for all $T>T_1$ (see Theorem 1, 2 in \cite{procaccia2019percolation} for details). The geometric properties for the infinite cluster, such as local uniqueness and order of chemical distance, for sufficiently large $T$ was later obtained in \cite{cai2020chemical}. However, as mentioned in Section 1.1 of \cite{procaccia2019percolation}, the uniqueness of phase transition and the existence/uniqueness of a critical fiber length $T_c$ remain open. Unlike the case for RI, this turns out to be a non-trivial problem, since it was recently proved that there is no global stochastic monotonicity with respect to $T$ for $\mathcal{FI}^{u,T}$, as shown in Theorem 1, \cite{cai2020non}. On the other hand, numerical tests in Section 5 of \cite{cai2020non} provide evidences on the existence and uniqueness of $T_c$, which was also conjectured in Conjecture 5 \cite{cai2020non} to be aymptotically an inverse linear function with respect to $u$.

(2) As in \cite{prevost2020percolation}, one may also fix $T$ and examine the evolution of FRI with respect to its intensity $u$. Note that FRI is by definition monotone with respect to $u$. So the question of interest here is whether there is always a non-trivial phase transition.  I.e., for a fixed $T$, we want to show FRI does not percolate for all sufficiently small $u$ and percolate for all sufficiently large $u$. \cite{prevost2020percolation} proved this for all $T$ small enough, and conjectured it can be extended to all $T\in (0,\infty)$. 

In this paper, we prove that, despite lacking global monotonicity, FRI is stochastically increasing with respect to $T$ for all $T\in (0,1)$, which implies the existence and uniqueness of $T_c$ for all sufficiently large $u$. Meanwhile, we also show that the upper bound of $T_c$ found by \cite{cai2020non, prevost2020percolation} in the high intensity regime is actually sharp, and give an exact asymptotic of $T_c$ as $u\to\infty$. Moreover, for the low intensity regime, we prove a polynomial lower bound for the phase diagram, which at the same time proves the conjecture on the global existence of a non-trivial phase transition with respect to $u$. Our proofs are largely based on the ``decoupling" methods first invented in \cite{sznitman2012decoupling}.

This paper is organized as follows. In Section \ref{sec2} we introduce precise definitions of FRI together with some necessary notations and preliminaries. We state our main results in Section \ref{sec3}.  Local monotonicity is shown in Section \ref{sec5}. And at last we estimate asymptotic of critical values in Section \ref{sec4}.

\section{Notations and preliminaries}\label{sec2}

\textbf{Some basic notations:} In this paper, we denote the $l^\infty$ distance and Euclidean distance by $|\cdot|$ and $|\cdot|_2$ respectively. We also denote the undirected edge set of $\mathbb{Z}^d$ by $\mathbb{L}^d$ (i.e. $\mathbb{L}^d:=\left\lbrace \{x,y\}:x,y\in \mathbb{Z}^d,|x-y|_2=1 \right\rbrace $). For any sets $A,B\subset \mathbb{Z}^d$, the $l^\infty$ distance between them is defined as $d(A,B):=\min\{|x-y|:x\in A,y\in B\}$. For finite subset $D\subset \mathbb{Z}^d$, let $\partial D:=\{x\in D:\exists y\in \mathbb{Z}^d\setminus D\ such\ that\ \{x,y\}\in \mathbb{L}^d \}$ be its inner boundary and $|D|$ be the cardinality of $D$, without casusing further confusion.

\textbf{Connection between two sets:} For sets $A,B\subset \mathbb{Z}^d$ and a collection of edges in $\mathbb{L}^d$ denoted by $E$, we say $A$ and $B$ are connected by $E$ (written by $A\xleftrightarrow[]{E} B$) iff there exists a sequence of vertices $(x_0,...,x_n)$ such that $x_0\in A$, $x_n\in B$ and that for any $ 0\le i\le n-1$, $\{x_i,x_{i+1}\}\in E$.

\textbf{Statements about constants:} we will use $c, c_1, c_2,...$ as local constants
(“local” means their values may vary according to contexts) and $C,C_1,C_2,...$ as global
constants (“global” means constants will keep their values in the whole paper).

\textbf{Random walks and relative stopping times:} We denote the law of simple random walks starting from $x$ on $\mathbb{Z}^d$ by $P_x$ and the law of geometrically killed simple random walks starting from $x$ with killing rate $\frac{1}{T+1}$ at each step by $P_x^{(T)}$.

For a random walk $\{X_i\}_{n=0}^{\infty}$ and $A\subset \mathbb{Z}^d$, define the hitting time and entrance time as $H_A=\min\{k\ge 0:X_k\in A\} $ and $\bar{H}_A=\min\{k\ge 1:X_k\in A\}$. Here we set $\min \emptyset=\infty$.

\textbf{Capacity with killing measure:} For any $K\subset \mathbb{Z}^d$, the escaping probability is defined as $Es_K^{(T)}(x):=P_x^{(T)}\left(\bar{H}_K=\infty\right) $. Then the capacity of $K$ with killing measure $\frac{1}{T+1}$ is $$cap^{(T)}\left(K\right):=2d*\sum_{x\in K}Es_K^{(T)}(x).$$
In particular, we have \begin{equation}
	cap^{(T)}\left(K\right)\le 2d*|K|.
\end{equation}

\textbf{Definitions of FRI:}  According to \cite{procaccia2019percolation}, FRI has two equivalent definitions. Denote the set of all finite nearest-neighbor paths on $\mathbb{Z}^d$ by $W^{\left[0,\infty \right) }$. Then $v^{(T)}:=\sum_{x\in \mathbb{Z}^d}\frac{2d}{T+1}P_x^{(T)}$ is a $\sigma-$finite measure  on $W^{\left[0,\infty \right) }$.
\begin{definition}\label{def1}
	For $0<u,T<\infty$, finitary random interlacements $\mathcal{FI}^{u,T}$ is the Poisson point
	process with intensity measure $u*v^{(T)}$. We denote the law of $\mathcal{FI}^{u,T}$ by $P^{u,T}$.
\end{definition}

\begin{definition}\label{def2}
	Let $\{N_x\}_{x\in \mathbb{Z}^d}\overset{i.i.d.}{\sim} Pois(\frac{2du}{T+1})$. For each site $x\in \mathbb{Z}^d$, start $N_x$ independent geometrically killed simple random walks with law $P_x^{(T)}$. Let $\mathcal{FI}^{u,T}$ be the point measure on $W^{\left[0,\infty \right) }$ composed of all the trajectories above starting from all $x\in \mathbb{Z}^d$.
\end{definition}
With a slight abuse of notations, for $\eta\in W^{\left[0,\infty \right) }$, we may write $\eta\in \mathcal{FI}^{u,T}$ if $\mathcal{FI}^{u,T}(\eta)=1$. In this paper, $\mathcal{FI}^{u,T}$ is also regarded as a bond percolation model on $\mathbb{L}^d$. I.e. we say an edge $e\in \mathbb{L}^d$ is open iff there exists a path in $\mathcal{FI}^{u,T}$ containing $e$ and then write $\mathcal{FI}^{u,T}(e)=1$ (otherwise, $\mathcal{FI}^{u,T}(e)=0$). For the simplicity of notations, for any $A,B\subset \mathbb{Z}^d$, we denote \begin{equation}
		A\xleftrightarrow[]{\mathcal{FI}^{u,T}}B:=A\xleftrightarrow[]{\left\lbrace e\in \mathbb{L}^d:\mathcal{FI}^{u,T}(e)=1\right\rbrace }B.
	\end{equation} 

We say $\mathcal{FI}^{u,T}$ percolates iff there exists an infinite connected cluster composed of open edges. It has been proved in Theorem 2, \cite{cai2020non} that $\mathcal{FI}^{u,T}$ contains at most one open infinite connected cluster.  
 
\textbf{FRI on a finite set:} Let $K$ be a finite subset of $\mathbb{Z}^d$. For each path $\eta_i$ in $\mathcal{FI}^{u,T}$, we denote the part of $\eta_i$ after intersecting $K$ by $\hat{\eta}_i^K$. Precisely, for $\eta_i=(\eta_i(0),...,\eta_i(l_i))$, define  $\hat{\eta}_i^K:=\left(\eta_i(H_K),...,\eta_i(l_i)\right)$ if $H_K<\infty$ and $\hat{\eta}_i^K:=\emptyset$ if $H_K=\infty$. By Lemma 2.1 of \cite{procaccia2019percolation}, $\sum_{i}\delta_{\hat{\eta}_i^K}$ has the same law as a Poisson point process with intensity measure $u*\sum_{x\in K}Es_K^{(T)}(x)P_x^{(T)}$. As a direct corollary, the number of paths intersecting $K$ in $\mathcal{FI}^{u,T}$ is a Poisson random variable with parameter $u*cap^{(T)}\left(K\right)$.

\textbf{Independence in FRI:} For disjoint sets $A_1,A_2,...,A_m\subset \mathbb{Z}^d$, if a sequence of events $E_1,E_2,...,E_m$ satisfying that for all $1\le i\le m$, $E_i$ only depends on the paths in $\mathcal{FI}^{u,T}$ starting from $A_i$. By Definition \ref{def2}, $E_1,E_2,...,E_m$ are independent.

Now we give the definition of critical values of FRI:
\begin{definition}
	For $u>0$,  $d\ge 3$, we define that $$T_c^-(u,d):=\sup\{T_0>0:\forall 0<T<T_0,\mathcal{FI}^{u,T}\ does\ not\  percolate\} $$ and $$T_c^+(u,d):=\inf\{T_0>0:\forall T>T_0,\mathcal{FI}^{u,T}\ percolates\}.$$
\end{definition}

\begin{remark}
	By Theorem 1 and 2 of \cite{procaccia2019percolation}, it has been peroved that for all $d\ge 3$, $0<T_c^-(u,d)\le T_c^+(u,d)<\infty $. 
\end{remark}

\section{Main results}\label{sec3}

In our first result, we show that for all $d\ge 3$ and $u>0$,  although $\mathcal{FI}^{u,T}$ does not enjoy monotonicity for larger $T$'s (see Theorem 1, \cite{cai2020non}), it is indeed stochastically increasing with respect to $T$ for all $T\in (0,1]$.

\begin{theorem}\label{mono}
	For any $u>0$, $0<T_1<T_2<\infty$ such that $T_1*T_2\le 1$, then $\mathcal{FI}^{u,T_2}$ stochasitcally dominates $\mathcal{FI}^{u,T_1}$. I.e. there is a coupling between $\mathcal{FI}^{u,T_1}$ and $\mathcal{FI}^{u,T_2}$ such that almost surely for any edge $e\in \mathbb{L}^d$, $\mathcal{FI}^{u,T_1}(e)\le \mathcal{FI}^{u,T_2}(e)$.
\end{theorem}
 
Combining Theorem \ref{mono} together with Theorem 3 (iv) and Proposition 2, \cite{cai2020non}, one can now have the existence and uniqueness of a critical fiber length for all sufficiently large $u$. 

\begin{corollary}\label{Tc}
For all $d\ge 3$, there is a $U_d<\infty$ such that for all $u\ge U_d$, 
$$
T_c^+(u,d)=T^-_c(u,d):=T_c(u,d).
$$
So we have $\mathcal{FI}^{u,T}$ percolates a.s. for all $T>T_c$, and does not percolates a.s. for all $T<T_c$.
\end{corollary}

Our next result provides the exact asymptotic of $T_c$ as $u\to\infty$, which gives an affirmative answer to Part 2, Conjecture 5, \cite{cai2020non}.
\begin{theorem}\label{prop1}
	For all $d\ge 3$, 
	\begin{equation}\label{5.3}
		\lim\limits_{u\to \infty}u*T_c(u,d)=\frac{-\log(1-p_d^{c})}{2},
	\end{equation}
	where $p_d^c$ is the critical value of Bernoulli bond percolation on $\mathbb{L}^d$.
\end{theorem}

\begin{remark}
In Theorem 4.2 \cite{cai2020non} and on Page 263 \cite{prevost2020percolation}, it has been shown that  
\begin{equation}
	\limsup\limits_{u\to \infty}u*T_c^+\le \frac{-\log(1-p_d^c)}{2}.
\end{equation}
So here we only need to prove a sharp lower bound for the subcritical phase (a partial result on the asymptotic order was given in Proposition 2, \cite{cai2020non}). 		
\end{remark}

Next we estimate the subcritical phase in the low intensity regime. 

\begin{theorem}\label{prop2}
	For any $d\ge 3$ and $\delta>0$, there exist constants $0<U_0(d,\delta)<1$ and $C_1(d,\delta)>0$ such that: for any $0<u\le U_0$, 
	\begin{equation}
	T_c^-\ge C_1u^{-\frac{1}{d-1}+\delta}	.
	\end{equation}
\end{theorem}

Combining Theorem \ref{prop2} and the supercritical estimates obtained in (v) of Theorem 3, \cite{cai2020non},  the following result can be summarized under log-log plot on the bounds we have about the phase transition.
\begin{corollary}
	When $d=3$, \begin{equation}
		0.5\le \liminf\limits_{u\to 0}\frac{\log(T_c^-)}{-\log(u)}\le	\limsup\limits_{u\to 0}\frac{\log(T_c^+)}{-\log(u)}\le 2;
	\end{equation}
	When $d\ge 4$, \begin{equation}
		\frac{1}{d-1}\le \liminf\limits_{u\to 0}\frac{\log(T_c^-)}{-\log(u)}\le	\limsup\limits_{u\to 0}\frac{\log(T_c^+)}{-\log(u)}\le 1.
	\end{equation}
\end{corollary}

Finally, Theorem \ref{prop1} together with Theorem \ref{prop2} also give an affirmative answer to the conjecture posed on Page 263, \cite{prevost2020percolation} on the global existence of a non-trivial phase transition with respect to $u$:
\begin{theorem}
For all $d\ge 3$ and $T\in (0,\infty)$, there is a $u_c=u_c(d,T)\in (0,\infty)$ such that $\mathcal{FI}^{u,T}$ percolates a.s. for all $u\in (u_c,\infty)$, and does not percolate a.s. for all $u\in (0,u_c)$.
\end{theorem}

\section{Local stochastic monotoncity}\label{sec5}

Before giving the proof of Theorem \ref{mono}, we introduce a new approach to construct $\mathcal{FI}^{u,T}$. Denote the collection of all directed nearest-neighbor edges by $\hat{\mathbb{L}}^d:=\left\lbrace x\to y:\{x,y\}\in \mathbb{L}^d \right\rbrace $. Let $\{N_{x\to y}\}_{x\to y\in \hat{\mathbb{L}}^d}\overset{i.i.d.}{\sim} Pois\left( \frac{uT}{(T+1)^2}\right) $, $\{Y_k\}_{k=1}^\infty\overset{i.i.d.}{\sim} Geo\left( \frac{T}{T+1}\right) $ and $\left\lbrace X_{\cdot}^{y,k} \right\rbrace_{k\in \mathbb{N}^+, y\in \mathbb{Z}^d}$ be a sequence of independent random walks with law $P_0^{(T)}$. Then we have

\begin{lemma}\label{lemma3}
	$$\bigcup\limits_{x\to y\in \hat{\mathbb{L}}^d}\bigcup\limits_{1\le k\le N_{x\to y},Y_k\ge 1}\left\lbrace \{x,y\}, \{X_i^{y,k},X_{i+1}^{y,k}\}, 0\le i \le Y_k-1\right\rbrace \overset{d}{=}\mathcal{FI}^{u,T}. $$
\end{lemma}
\begin{proof}
	For any $x\in \mathbb{Z}^d$ and $x\to y\in \hat{\mathbb{L}}^d$, by Definition \ref{def2} and property of Poisson point process, we know the number of paths starting from $x$ with length $\ge 1$ and the first step $x\to y$ is a Poisson variable with parameter $\frac{T}{T+1}*\frac{1}{2d}*\frac{2du}{T+1}=\frac{uT}{(T+1)^2}$. In addition, by the memoryless property of the geometric distribution, the remaining part of each path after the first step removed is still a geometrically killed random walk with killing rate $\frac{1}{T+1}$. Then Lemma \ref{lemma3} follows.
\end{proof}

Now we give the proof of Theorem \ref{mono} as follows:

\begin{proof}[Proof of Theorem \ref{mono}:]
	When $T_1<T_2$, $0<T_1*T_2<1$, we define three independent sequences of random variables: $$\left\lbrace N_{x\to y}^{(1)} \right\rbrace\overset{i.i.d.}{\sim}Pois\left( \frac{uT_1}{(T_1+1)(T_2+1)}\right) , $$ $$\left\lbrace N_{x\to y}^{(2)} \right\rbrace\overset{i.i.d.}{\sim}Pois\left( \frac{uT_1(T_2-T_1)}{(T_1+1)^2(T_2+1)}\right) $$ and $$\left\lbrace N_{x\to y}^{(3)} \right\rbrace\overset{i.i.d.}{\sim}Pois\left( \frac{u(T_2-T_1)(1-T_1T_2)}{(T_1+1)^2(T_2+1)^2}\right) . $$
	
	Meanwhile, it's elementary to construct $$\left\lbrace \left( Y_k^{(T_1)},Y_k^{(T_2)}\right)  \right\rbrace_{k=1}^{\infty}\overset{i.i.d.}{\sim} \left( Geo\left( \frac{T_1}{T_1+1}\right) ,Geo\left( \frac{T_2}{T_2+1}\right) \right)$$ such that for any $k\ge 1$, $P\left(Y_k^{(T_1)}\le Y_k^{(T_2)} \right)=1$.
	
	Note that for any $x\to y$, $N_{x\to y}^{(1)}+N_{x\to y}^{(2)}\sim Pois\left( \frac{uT_1}{(T_1+1)^2}\right) $ and $N_{x\to y}^{(1)}+N_{x\to y}^{(2)}+N_{x\to y}^{(3)}\sim Pois\left( \frac{uT_2}{(T_2+1)^2}\right)$. By Lemma \ref{lemma3}, we have  \begin{equation}\label{5.1}
		\bigcup\limits_{x\to y\in \hat{\mathbb{L}}^d}\bigcup\limits_{1\le k\le N^{(1)}_{x\to y}+N^{(2)}_{x\to y},Y^{(T_1)}_k\ge 1}\left\lbrace \{x,y\}, \{X_i^{y,k},X_{i+1}^{y,k}\}, 0\le i \le Y^{(T_1)}_k-1\right\rbrace \overset{d}{=}\mathcal{FI}^{u,T_1}.
	\end{equation}
	
	 Note that for any $k\ge 1$, if $Y^{(T_1)}_k\ge 1$, then $Y^{(T_2)}_k\ge 1$. Use Lemma \ref{lemma3} again, \begin{equation}\label{4.2}
	 	\bigcup\limits_{x\to y\in \hat{\mathbb{L}}^d}\bigcup\limits_{1\le k\le N^{(1)}_{x\to y}+N^{(2)}_{x\to y}+N^{(3)}_{x\to y},Y^{(T_2)}_k\ge 1}\left\lbrace \{x,y\}, \{X_i^{y,k},X_{i+1}^{y,k}\}, 0\le i \le Y^{(T_2)}_k-1\right\rbrace \overset{d}{=}\mathcal{FI}^{u,T_2}.
	 \end{equation}
 
 Finally, by comparing the LHS's of (\ref{5.1}) and (\ref{4.2}), we get the stochasitc domination in Theorem \ref{mono}.
\end{proof}

\section{Asymptotic of critical values}\label{sec4}
In this section, we prove Theorem \ref{prop1} and Theorem \ref{prop2}. Before presenting the proof, we need to first introduce some notations according to \cite{rodriguez2013phase} in order to implement the renormalization arguement.

\begin{enumerate}
	\item Let $L_0$ and $l_0$ be positive integers. For $n\ge 1$, let $L_n=l_0^n*L_0$ and $\mathbb{L}_n=L_n*\mathbb{Z}^d$.
	
	\item Set $B_{n,x}:=x+\left(\left[0,L_n\right)\cap \mathbb{Z}\right)^d$ and $\widetilde{B}_{n,x}=\bigcup\limits_{y\in \mathbb{L}_n: d(B_{n,y},B_{n,x})\le 1 }B_{n,y}$.
	
	\item Let $\mathcal{I}_n=\{n\}\times \mathbb{L}_n$. For any $(n,x)\in \mathcal{I}_n$, write 
	$$\mathcal{H}_1(n,x)=\left\lbrace (n-1,y)\in \mathcal{I}_{n-1}:B_{n-1,y}\subset B_{n,x},B_{n-1,y}\cap \partial B_{n,x} \neq \emptyset \right\rbrace  ,$$
	$$\mathcal{H}_2(n,x)=\left\lbrace (n-1,y)\in \mathcal{I}_{n-1}:B_{n-1,y}\cap \{z\in \mathbb{Z}^d:d(z,B_{n,x})=\lfloor \frac{L_n}{2} \rfloor \} \neq \emptyset \right\rbrace.  $$
	
	\item For $x\in \mathbb{L}^d$, $n\ge 0$, let \begin{equation}\label{Lambda}
		\begin{split}
			\Lambda_{n,x}=\{\mathcal{T}\subset \bigcup\limits_{k=0}^{n}\mathcal{I}_k:&\mathcal{T}\cap\mathcal{I}_n=(n,x)\ and\ \forall (k,y)\in \mathcal{T}\cap \mathcal{I}_k, 0<k\le n, has\ two\\
			& descendants\ (k-1,y_i(k,y))\in \mathcal{H}_i(k,y),i=1,2\ such\ that\\
			& \mathcal{T}\cap \mathcal{I}_{k-1}=\bigcup\limits_{(k,y)\in \mathcal{T}\cap \mathcal{I}_{k-1}}\{(k-1,y_1(k,y)),(k-1,y_2(k,y))\}      \}.    
		\end{split}
	\end{equation}
\end{enumerate}

By (2.8) of \cite{rodriguez2013phase}, one has \begin{equation}\label{5.2}
	|\Lambda_{n,x}|\le (c_0(d)*l_0^{2(d-1)})^{2^n}.
\end{equation}	

Based on these settings, we can do decompositions on the events and then estimate their probabilities by choosing proper $L_0$ and $l_0$. Roughly speaking, we need to select $L_0$ to control the $0-$level event and select $l_0$ to guarantee the ``almost indepedence'' between trajectories in different boxes, according to $u$ and $T$.

Now we give the proof of Theorem \ref{prop1}.

\subsection{Proof of Theorem \ref{prop1}} 

Recall that in Theorem 3.(iii) of \cite{cai2020non} and on Page 263 of \cite{prevost2020percolation}, it has been shown that  
\begin{equation}
	\limsup\limits_{u\to \infty}u*T_c^+\le \frac{-\log(1-p_d^c)}{2}.
\end{equation}		
Now it's sufficient to prove that: for any $\epsilon>0$, there exists $U'(d,\epsilon)>0$ such that for any $u>U'$ and $T>0$ satisfying 
\begin{equation}\label{5.4}
	u*T\le  \frac{-\log(1-p_d^c)}{2}-\epsilon,
\end{equation}
$\mathcal{FI}^{u,T}$ does not percolate, from which we deduce
\begin{equation}
		\liminf\limits_{u\to \infty}u*T_c^-\ge \frac{-\log(1-p_d^c)}{2}.
	\end{equation}

   Assume $u$ is sufficiently large while $T$ satisfies (\ref{5.4}). Take $L_0=\lfloor u^{\frac{1}{2d}}\rfloor>10$ and $l_0=10$. Let $B^*_{0,0}=\{y\in \mathbb{Z}^d:d(\{y\},\widetilde{B}_{0,0})\le 1\}$ and recall the notation $\hat{\eta}^K$ in Section \ref{sec2}, then we have
	\begin{equation}\label{57}
		\begin{split}
			&P^{u,T}\left(\exists \eta \in \mathcal{FI}^{u,T}\ such\ that\ \{\eta(i),\eta(i+1)\}\subset \widetilde{B}_{0,0}\ for\ some\ i\ge 1\right) \\
			\le &P^{u,T}\left(\exists \eta\in \mathcal{FI}^{u,T}\ such\ that\ the\ length\ of\ \hat{\eta}^{B^*_{0,0}}\ \ge 2\right)\\
			=&1-\exp(-u*cap^{(T)}(B^*_{0,0})*\left(\frac{T}{T+1} \right)^2 )\\
			\le &u*cap^{(T)}(B^*_{0,0})*\left(\frac{T}{T+1} \right)^2\\
			\le &c*u^{1.5}*T^2\le c'*u^{-0.5}.
		\end{split}
	\end{equation}

	Let $\hat{\mathcal{FI}}^{u,T}:=\left\lbrace e\in \mathbb{L}^d:e=\{\eta(0),\eta(1)\},\eta\in \mathcal{FI}^{u,T} \right\rbrace $. By Definition \ref{def2}, it's easy to see that $\hat{\mathcal{FI}}^{u,T}$ has the same distribution as the collection of open edges of a Bernoulli bond percolation with parmeter $1-\exp(-\frac{uT}{(T+1)^2})<p_d^c$. 
	
	By Theorem 6.1 of \cite{grimmett2013percolation}, we have \begin{equation}\label{47}
		P^{u,T}\left(B_{0,0}\xleftrightarrow[]{\hat{\mathcal{FI}}^{u,T}} \partial \widetilde{B}_{0,0} \right)\le (2L_0+1)^d*P^{u,T}\left(0\xleftrightarrow[]{\hat{\mathcal{FI}}^{u,T}} \partial B_{0,0}  \right)\le (2L_0+1)^d*e^{-cL_0}.
	\end{equation}
	
	
	Note that when the event in the LHS of (\ref{57}) does not occur, $\left\lbrace B_{0,0}\xleftrightarrow[]{\mathcal{FI}^{u,T}} \partial \widetilde{B}_{0,0}\right\rbrace $ is equivalent
		 to $\left\lbrace B_{0,0}\xleftrightarrow[]{\hat{\mathcal{FI}}^{u,T}} \partial \widetilde{B}_{0,0}\right\rbrace $. By (\ref{57}) and $(\ref{47})$, \begin{equation}\label{4.8}
    	P^{u,T}\left(B_{0,0}\xleftrightarrow[]{\mathcal{FI}^{u,T}} \partial \widetilde{B}_{0,0} \right)\le c'*u^{-0.5}+(2L_0+1)^d*e^{-cL_0}.
    \end{equation}
	
	For any $n\ge 0$ and $x\in \mathbb{L}_n$, we write that \begin{equation}
		A_{n,x}:=\left\lbrace B_{n,x}\xleftrightarrow[]{\mathcal{FI}^{u,T}} \partial \widetilde{B}_{n,x}  \right\rbrace.
	\end{equation}
	
	Recall the notation $\Lambda_{n,x}$ in (\ref{Lambda}). For any $\mathcal{T}\in \Lambda_{n,x}$, like (2.13) of \cite{rodriguez2013phase}, we write \begin{equation}
		A_{\mathcal{T}}:=\bigcap\limits_{(0,y)\in \mathcal{T}\cap \mathcal{I}_0}A_{0,y}.
	\end{equation}
	Similar to (2.14) of \cite{rodriguez2013phase}, we have \begin{equation}\label{5.11}
		A_{n,x}\subset \bigcup\limits_{\mathcal{T}\in \Lambda_{n,x}}A_{\mathcal{T}}.
	\end{equation}
	
Here we need a decoupling inequality, which is parallel to Lemma 5.4, \cite{cai2020chemical}. Let $\hat{B}_{n,x}=\{y:d(\{y\},\widetilde{B}_{n,x})\le L_{n+1}\}$ and $F_{n,x}:=\bigcap\limits_{\eta \in \mathcal{FI}^{u,T}, \eta(0)\in \left(\hat{B}_{n,x}\right)^c } \{\eta \cap \widetilde{B}_{n,x}=\emptyset\}.$

If $F_{n,x}$ does not occur, there must exist a path $\eta\in\mathcal{FI}^{u,T}$ such that $\eta \cap \partial \hat{B}_{n,x}\neq \emptyset$ and $\eta \cap \partial \widetilde{B}_{n,x} \neq \emptyset$. Recalling the construction of FRI on a finite set in Section \ref{sec2} (here we take $K=\partial\hat{B}_{n,x}\cup \partial \widetilde{B}_{n,x} $), the number of such paths is a Poisson random variable with parameter \begin{equation}\label{5.13}
	\begin{split}
		&u*\left[ \sum_{y\in \partial\hat{B}_{n,x} }Es_{K}^{(T)}(y)*P_y^{(T)}\left(H_{\widetilde{B}_{n,x}}<\infty \right)+\sum_{y\in  \partial \widetilde{B}_{n,x} }Es^{(T)}_{K}(y)*P_y^{(T)}\left(H_{\partial \hat{B}_{n,x}}<\infty\right)  \right]\\
		\le &c*u*(L_{n+1})^{d}*\left(1-\frac{1}{T+1}\right)^{L_{n+1}}\le c_1u^{1.5}10^{nd}e^{-c_2L_{n+1}}.
	\end{split}
\end{equation}

By $l_0=10$ and (\ref{5.13}), when $u$ is sufficiently large, we have: for any $n\ge 0$,
\begin{equation}\label{5.12}
	P^{u,T}\left(F_{n,x}\right)\ge \exp(-c_1u^{1.5}10^{nd}e^{-c_2L_{n+1}})\ge 1-e^{-c_3*L_0*2^n}.
\end{equation}	
		
Assume that $\mathcal{T}\cap \mathcal{I}_{n-1}=\{(n-1,y_1), (n-1,y_2)\}$, where $y_1,y_2\in \mathbb{L}_{n-1}$. Since $L_0>10$, $\hat{B}_{n-1,y_1} \cap \hat{B}_{n-1,y_2}=\emptyset$. We denote that $\mathcal{T}_i=\{(m,y)\in \mathcal{T}:y\in \widetilde{B}_{n-1,y_i}\}$, $i\in \{1,2\}$. We also write \begin{equation}
		A_{\mathcal{T}_i}=\bigcap\limits_{(0,y)\in \mathcal{T}_i\cap \mathcal{I}_0}A_{0,y}\ and\ \hat{A}_{\mathcal{T}_i}=\bigcap\limits_{(0,y)\in \mathcal{T}_i\cap \mathcal{I}_0}\left\lbrace B_{0,y}\xleftrightarrow[]{\{\eta\in \mathcal{FI}^{u,T}: \eta(0)\in \hat{B}_{n-1,y_i}\}} \partial \widetilde{B}_{0,y}  \right\rbrace. 
	\end{equation}
	Note that $\hat{A}_{\mathcal{T}_1}$ and $\hat{A}_{\mathcal{T}_2}$ are independent and for $i\in \{1,2\}$, $\hat{A}_{\mathcal{T}_i}\subset A_{\mathcal{T}_i}$. In addition, if $F_{n-1,y_1}$ and $F_{n-1,y_2}$ both occur, events $\hat{A}_{\mathcal{T}}$ and $\hat{A}_{\mathcal{T}_1}\cap \hat{A}_{\mathcal{T}_2}$ will be equivalent.

	 By induction and (\ref{5.12}), similar to (5.26) of \cite{cai2020chemical}, we have 
	 \begin{equation}\label{5.15}
		\begin{split}
			P^{u,T}\left(A_{\mathcal{T}}\right)+2e^{-c_3*L_0*2^n}\le & P^{u,T}\left(A_{\mathcal{T}}\cap F_{n-1,y_1}\cap F_{n-1,y_2} \right)+4e^{-c_3*L_0*2^n}\\
			\le &P^{u,T}\left(\hat{A}_{\mathcal{T}_1}\cap \hat{A}_{\mathcal{T}_2} \right)+4e^{-c_3*L_0*2^n}\\
			\le &\left(P^{u,T}\left(A_{\mathcal{T}_1}\right)+2e^{-c_3*L_0*2^{n-1}} \right) *\left(P^{u,T}\left(A_{\mathcal{T}_2}\right)+2e^{-c_3*L_0*2^{n-1}} \right)\\
			\le &...\le \left(P^{u,T}\left(A_{0,0}\right)+2e^{-c_3*L_0} \right)^{2^n}.
		\end{split}
	\end{equation}

	Combine (\ref{5.2}), (\ref{4.8}), (\ref{5.11}) and (\ref{5.15}), \begin{equation}\label{5.16}
		P^{u,T}\left(A_{n,0} \right)\le \left(c_0*l_0^{2(d-1)} \right)^{2^n}*\left(c'*u^{-0.5}+(2L_0+1)^d*e^{-cL_0}+2e^{-c_3*L_0} \right)^{2^n}.   
	\end{equation}

    Since $L_0=\lfloor u^{\frac{1}{2d}}\rfloor$ and $l_0=10$, the sufficiently large $u$ will also give 
 \begin{equation}
		c_0*l_0^{2(d-1)}*\left(c'*u^{-0.5}+(2L_0+1)^d*e^{-cL_0}+2e^{-c_3*L_0}\right)<0.5.
	\end{equation}
Therefore, as $n\to \infty$, $P^{u,T}\left(A_{n,0} \right)$ converges to $0$. This implies  $P^{u,T}\left(0\xleftrightarrow[]{\mathcal{FI}^{u,T}} \partial\widetilde{B}_{n,0} \right)$ also converges to $0$, since $P^{u,T}\left(0\xleftrightarrow[]{\mathcal{FI}^{u,T}} \partial\widetilde{B}_{n,0} \right)\le P^{u,T}\left(A_{n,0} \right)$. So $\mathcal{FI}^{u,T}$ does not percolate in this case.
	
	In conclusion, \begin{equation}\label{5.19}
		\frac{-\log(1-p_d^c)}{2}\le \liminf\limits_{u\to \infty}u*T_c^-\le \limsup\limits_{u\to \infty}u*T_c^+\le \frac{-\log(1-p_d^c)}{2}.
	\end{equation}
	Then we get Theorem \ref{prop1} by (\ref{5.19}) and Corollary \ref{Tc}.
\qed

\subsection{Proof of Theorem \ref{prop2}:}

Before we give the proof of Theorem \ref{prop2}, we need an estimate on the diameter of the range of geometrically killed random walks.
\begin{lemma}\label{lemma2}
	Recall that $\{X_n^{(T)}\}$ is a geometrically killed random walk with law $P_0^{(T)}$. Then there exists $c_1,c_2>0$ such that for any $L>0$,  
	\begin{equation}
		P_0^{(T)}\left[\max_{0\le i\le \infty}|X_i^{(T)}|\ge L \right]\le c_1e^{-c_2*T^{-\frac{1}{3}}*L^{\frac{2}{3}}}.
	\end{equation}
\end{lemma}
\begin{proof}
	By Theorem 1.5.1 of \cite{lawler2013intersections}, for any positive integer $m$, we have \begin{equation}\label{5.25}
		\begin{split}
			P_0^{(T)}\left[\max_{0\le i\le \infty}|X_i^{(T)}|\ge L \right]\le& (1-\frac{1}{T+1})^{m}+P_0\left[\max_{0\le i\le m}|X_i|\ge L \right]\\
			\le & \exp(-c*T^{-1}m)+c'*\exp(-\frac{L}{\sqrt{m}}).
		\end{split}
	\end{equation}
Take $m=\lfloor(TL)^{\frac{2}{3}}\rfloor$ in (\ref{5.25}) and then we get Lemma \ref{lemma2}.
\end{proof}

By Theorem \ref{mono}, for any fixed $u>0$, $\mathcal{FI}^{u,T}$ is stochastically increasing on $T\in \left(0,1\right] $. Thus, it's sufficient to confirm: for $d\ge 3$ and $\epsilon>0$, there exists $\bar{c}(d,\epsilon)>0$ and $U''(d,\epsilon)>0$ such that for any $u<U''$ and $T\ge 1$ satisfying \begin{equation}\label{5.21}
		u\le \bar{c}(T*\left(\log(T+1)\right) ^{3+\epsilon})^{-(d-1)},
	\end{equation}
	$\mathcal{FI}^{u,T}$ does not percolate. In fact, the result proved here is sightly stronger than Theorem \ref{prop2} and the RHS of (\ref{5.21}) can be replaced by a polynomial of $T$ to make the proof a little shorter.
	
	We use the same approach as in the proof of Theorem \ref{prop1} with different $L_0$ and $l_0$. Precisely, we set $L_0=10$, $l_0=\lfloor\left(c'T*(\log(T+1))^{3+\epsilon}\right) ^{0.5}\rfloor$, where $c'$ will be determined later. Note that the number of paths intersecting $B_{0,0}$ in $\mathcal{FI}^{u,T}$ is a Poisson random variable with parameter $u*cap^{(T)}(B_{0,0})$. Therefore, \begin{equation}\label{5.24}
		\begin{split}
			P^{u,T}\left(0\xleftrightarrow[]{\mathcal{FI}^{u,T}} \partial B_{0,0} \right)\le& P^{u,T}\left(there\ exists\ at\ least\ one\ path\ intersecting\ B_{0,0} \right) \\
			\le& 1-\exp(-u*cap^{(T)}(B_{0,0}))\le c''*u.
		\end{split}
	\end{equation}
	
	 For any $k\in \mathbb{N}^+$, $0\le r< L_{n+1}$ and $L=k*L_{n+1}+r$, by Lemma \ref{lemma2}, we have \begin{equation}\label{422}
	 	\begin{split}
	 		P_0^{(T)}\left(\max_{0\le i\le \infty}|X_i^{(T)}|\ge L \right)\le & c_1\exp(-c_2'(d)k^{\frac{2}{3}}(c')^{\frac{1}{3}(n+1)}T^{\frac{1}{3}n}\left( \log(T+1)\right)^{\frac{1}{3}(3+\epsilon)(n+1)}).
	 	\end{split}
	 \end{equation}
	 
	 Recall the notation $F_{n,x}$ in the proof of Theorem \ref{prop1}. By (\ref{422}), when the constant $c'(d,\epsilon)$ is sufficiently large, for $u>0$ and $T\ge 1$ satisfying (\ref{5.21}), 
	 \begin{equation}\label{4.23}
	 	\begin{split}
	 		P^{u,T}\left((F_{n,x})^c\right)\le&\sum_{y\in (\hat{B}_{n,0})^c}P^{u,T}\left(there\ is\ no\ path\ starting\ from\ y\ and\ intersecting\ \widetilde{B}_{n,x}\right)\\
	 		\le &\sum_{L\ge L_{n+1}}c*L^{d-1}*\left[1- \exp(-\frac{2du}{T+1}*P_0^{(T)}\left( \max_{0\le i\le \infty}|X_i^{(T)}|\ge L\right) )\right]\\
	 		\le &\sum_{k=1}^{\infty}c\left[(k+1)L_{n+1}\right]^{d-1}*\frac{2du}{T+1}*L_{n+1}\\
	 		&*\exp(-c_2'k^{\frac{2}{3}}(c')^{\frac{1}{3}(n+1)}T^{\frac{1}{3}n}\left( \log(T+1)\right)^{\frac{1}{3}(3+\epsilon)(n+1)})\\
	 		\le&c*u*(L_{n+1})^d*\exp(-c_2''(d)*(c')^{n+1}*T^{\frac{1}{3}n}*\left( \log(T+1)\right) ^{\frac{1}{3}(3+\epsilon)(n+1)})\\
	 		\le &\exp(-c_2'''(d)*c'*\left( \log(T+1)\right) ^{\frac{1}{3}(3+\epsilon)}*2^n).
	 	\end{split}
	 \end{equation} 

	Similar to (\ref{5.16}), by (\ref{5.24}) and (\ref{4.23}), we have \begin{equation}
		\begin{split}
			P^{u,T}\left(0\xleftrightarrow[]{\mathcal{FI}^{u,T}} \partial \widetilde{B}_{n,0} \right)\le& \left[c_0l_0^{2(d-1)}*\left(2\exp(-c_2'''*c'*\left( \log(T+1)\right) ^{\frac{1}{3}(3+\epsilon)})+c''*u\right)\right] ^{2^n}. 
		\end{split}
	\end{equation}

	 Take $\bar{c}=(4c_0c'')^{-1}$ and $c'(d,\epsilon)>0$ large enough such that for any $T\ge 1$, \begin{equation}
	 	c_0l_0^{2(d-1)}*2\exp(-c_2'''*c'*\left( \log(T+1)\right) ^{\frac{1}{3}(3+\epsilon)})<0.25.
	 \end{equation} 
 
 Therefore, when $u$ is sufficiently small and $T\ge 1$ satisfies (\ref{5.21}), $$\lim\limits_{n\to \infty}P^{u,T}\left(0\xleftrightarrow[]{\mathcal{FI}^{u,T}} \partial \widetilde{B}_{n,0} \right)= 0$$ and thus $\mathcal{FI}^{u,T}$ does not percolate.
\qed

\section*{Acknowledgement}

The authors would like to thank Drs. Xinyi Li, Eviatar B. Procaccia, and Ron Rosenthal for fruitful discussions. The authors would like to also thank Dr. Eviatar B. Procaccia for allowing us to use his LaTeX macros.

\bibliographystyle{plain}
\bibliography{ref}

\begin{thebibliography}{10}

\bibitem{bowen2019finitary}
L.~Bowen.
\newblock Finitary random interlacements and the gaboriau--lyons problem.
\newblock {\em Geometric and Functional Analysis}, 29(3):659--689, 2019.

\bibitem{cai2020chemical}
Z.~Cai, X.~Han, J.~Ye, and Y.~Zhang.
\newblock On chemical distance and local uniqueness of a sufficiently
  supercritical finitary random interlacement.
\newblock {\em arXiv preprint arXiv:2009.04044}, 2020.

\bibitem{cai2020non}
Z.~Cai, Y.~Xiong, and Y.~Zhang.
\newblock On (non-)monotonicity and phase diagram of finitary random
  interlacement.
\newblock {\em Entropy}, 23(1):69, 2021.

\bibitem{vcerny2012internal}
J.~{\v{C}}ern{\`y} and S.~Popov.
\newblock On the internal distance in the interlacement set.
\newblock {\em Electronic Journal of Probability}, 17, 2012.

\bibitem{grimmett2013percolation}
G.R. Grimmett.
\newblock {\em Percolation}.
\newblock Grundlehren der mathematischen Wissenschaften. Springer Berlin
  Heidelberg, 2013.

\bibitem{lawler2013intersections}
G.F. Lawler.
\newblock {\em Intersections of random walks}.
\newblock Springer Science \& Business Media, 2013.

\bibitem{prevost2020percolation}
A.~Pr{\'e}vost.
\newblock {\em Percolation for the Gaussian free field and random
  interlacements via the cable system}.
\newblock PhD thesis, Universit{\"a}t zu K{\"o}ln, 2020.

\bibitem{procaccia2019percolation}
E.B. Procaccia, J.~Ye, and Y.~Zhang.
\newblock Percolation for the finitary random interlacements.
\newblock {\em arXiv preprint arXiv:1908.01954}, 2019.

\bibitem{rodriguez2013phase}
P.F. Rodriguez and A.S. Sznitman.
\newblock Phase transition and level-set percolation for the gaussian free
  field.
\newblock {\em Communications in Mathematical Physics}, 320(2):571--601, 2013.

\bibitem{sznitman2010vacant}
A.S. Sznitman.
\newblock Vacant set of random interlacements and percolation.
\newblock {\em Annals of mathematics}, pages 2039--2087, 2010.

\bibitem{sznitman2012decoupling}
A.S. Sznitman.
\newblock Decoupling inequalities and interlacement percolation on {$G \times
  \mathbb{Z}$}.
\newblock {\em Inventiones mathematicae}, 187(3):645--706, 2012.

\end{thebibliography}
\end{document}